\documentclass[a4paper]{article}

\usepackage{amsmath, amsthm, amsfonts, amssymb, accents}
\usepackage{graphicx,color}
\usepackage{srcltx}
\usepackage{dsfont}

\theoremstyle{plain}
\newtheorem{theorem}{Theorem}[section]

\newtheorem{lemma}[theorem]{Lemma}

\theoremstyle{remark}
\newtheorem{remark}{Remark}[section]

\numberwithin{equation}{section}

\def\Om{\Omega}

\def\e{\varepsilon}
\def\g{\gamma}
\def\G{\Gamma}
\def\l{\lambda}
\def\p{\partial}
\def\D{\Delta}

\def\k{\varkappa}
\def\E{\mbox{\rm e}}
\def\a{\alpha}
\def\b{\beta}

\def\Odr{\mathcal{O}}
\def\H{W_2}

\def\Hinf{W_\infty}

\def\di{\,d}

\def\I{\mathrm{I}}
\def\iu{\mathrm{i}}

\def\Op{\mathcal{H}}

\DeclareMathOperator{\RE}{Re}
\DeclareMathOperator{\IM}{Im}

\newcounter{assumption}
\begin{document}

\title{\textbf{Eigenvalues collision for $\mathcal{PT}$-symmetric waveguide}}
\author{Denis Borisov$^a$}
\date{\small
\begin{center}
\begin{quote}
{\it
Institute of Mathematics of Ufa Scientific Center of RAS, Chernyshevskogo, str. 112, 450008, Ufa, Russian Federation \&
Bashkir State Pedagogical University,
October St.~3a, 450000, Ufa,
Russian Federation; E-mail: \texttt{borisovdi@yandex.ru}
}
\end{quote}
\end{center}}

 \maketitle

\begin{abstract}
We consider a model of planar $\mathcal{PT}$-symmetric waveguide and study the phenomenon of the eigenvalues collision under the perturbation of boundary conditions. This phenomenon was discovered numerically in previous works. The main result of this work is an analytic explanation of this phenomenon. 
\end{abstract}

\section{Introduction and main results}

In this paper we study a problem in the theory of $\mathcal{PT}$-symmetric operators which is being studied rather intensively after pioneering works \cite{BB}, \cite{M1}, \cite{M2}, \cite{M3}, \cite{M4}, \cite{Z1}, \cite{Z2}, \cite{Z3}, \cite{Z4}, \cite{B}. Our model is introduced as follows.

Let $x=(x_1,x_2)$ be Cartesian coordinates in $\mathds{R}^2$, $\Om$ be the strip $\{x: -d<x_2<d\}$, $d>0$, $\a=\a(x_1)$ be a function in $\Hinf^1(\mathds{R})$. We consider the operator $\Op_\a$ in $L_2(\Om)$ acting as $\Op_\a u=-\D u$ on the functions $u\in\H^2(\Om)$ satisfying non-Hermitian boundary conditions
\begin{equation}\label{1.1}
\left(\frac{\p\hphantom{x}}{\p x_2}+\iu\a\right)u=0\quad\text{on}\quad\p\Om.
\end{equation}
It was shown in \cite{IEOP08} that this operator is $m$-sectorial, densely defined, and $\mathcal{PT}$-symmetric, namely,
\begin{equation}\label{1.2}
\mathcal{PT}\Op_\a=\Op_\a\mathcal{PT},
\end{equation}
where $(\mathcal{P}u)(x)=u(x_1,-x_2)$, and $\mathcal{T}$ is the operator of complex conjugation, $\mathcal{T}u=\overline{u}$. It was also proven in \cite{IEOP08} that
\begin{equation}\label{1.3}
\Op_\a^*=\Op_{-\a},\quad \Op_\a^*=\mathcal{T}\Op_\a\mathcal{T}=\mathcal{P}\Op_\a\mathcal{P}.
\end{equation}

A non-trivial question related to $\Op_\a$ is the behavior of its eigenvalues. As $\a(x_1)$ is a small regular localized perturbation of a constant function, in \cite{IEOP08} there were obtained sufficient conditions for existence and absence of  isolated eigenvalues near the threshold of the essential spectrum. Similar results for both regularly and singularly perturbed models were obtained in \cite{IMM12}, \cite{UMZh14}, \cite{AP07}, \cite{AsAn12}, \cite{KS}.

Numerical experiments performed in \cite{KS}, \cite{KT} brought a very non-trivial picture of the eigenvalues distribution. One of the interesting phenomenon discovered numerically in \cite{KS}, \cite{KT} was the eigenvalues collision. Namely, let $t\in \mathds{R}$ is a parameter, then as $t$ increases, the operator $\Op_{t\a}$ can have two simple real isolated eigenvalues meeting at some point. Then two cases are possible. In the first of them, these eigenvalues stay real as $t$ increases and they just pass along the real line. In the second case the eigenvalues become complex as $t$ increases and they are located symmetrically w.r.t. the real axis. The present paper is devoted to the analytic study of the described phenomenon.

Suppose $\l_0\in \mathds{R}$  is an isolated eigenvalue of $\Op_\a$, $\e$ is a small real parameter, $\b\in\Hinf^2(\mathds{R})$ is a some function. Denote $\G_\pm:=\{x: x_2=\pm d\}$. Our first main result describes the case when $\l_0$ is an eigenvalue of geometric multiplicity two.

\begin{theorem}\label{th1}
Assume $\l_0\in\mathds{R}$ is a double eigenvalue, $\psi_0^\pm$ are the associated eigenfunctions satisfying
\begin{equation}\label{1.9}
(\psi_0^\pm,\mathcal{T}\psi_0^\pm)_{L_2(\Om)}=1,\quad (\psi_0^+,\mathcal{T}\psi_0^-)_{L_2(\Om)}=0.
\end{equation}
Suppose also
\begin{gather}\label{1.4a}
(b_{11}-b_{22})^2+4b_{12}^2\not=0,
\\
\begin{aligned}
&b_{11}=\iu\int\limits_{\G_+} \b(\psi_0^+)^2\di x_1-\iu \int\limits_{\G_-} \b(\psi_0^+)^2\di x_1,
\\
&b_{22}=\iu\int\limits_{\G_+} \b(\psi_0^-)^2\di x_1-\iu \int\limits_{\G_-} \b(\psi_0^-)^2\di x_1,
\\
&b_{12}=\iu\int\limits_{\G_+} \b\psi_0^+\psi_0^-\di x_1-\iu \int\limits_{\G_-} \b\psi_0^+\psi_0^-\di x_1.
\end{aligned}\label{1.4b}
\end{gather}
Then for all sufficiently small $\e$ the operator $\Op_{\a+\e\b}$ has two simple isolated eigenvalues $\l_\e^\pm$ converging to $\l_0$ as $\e\to0$. These eigenvalues are holomorphic w.r.t. $\e$ and the first terms of their Taylor series are
\begin{align}
&\l_\e^\pm=\l_0+\e \l_1^\pm+\Odr(\e^2), \label{1.4}
\\
&\l_1^\pm=\frac{1}{2} (b_{11}+b_{22}) \pm \frac{1}{2} \big((b_{11}-b_{22})^2+4b_{12}^2\big)^{1/2}.\label{1.5}
\end{align}
\end{theorem}

The second main result is devoted to the case when the geometric multiplicity of $\l_0$ is one but the algebraic multiplicity is two.

\begin{theorem}\label{th2}
Let $\l_0\in\mathds{R}$ be a simple eigenvalue of $\Op_\a$, $\psi_0$ be the associated eigenfunction. Assume that the equation
\begin{equation}\label{1.6}
(\Op_\a-\l_0)\phi_0=\psi_0
\end{equation}
is solvable and there exists a solution satisfying
\begin{equation}\label{1.7}
(\phi_0,\mathcal{T}\psi_0)_{L_2(\Om)}\not=0,\quad (\phi_0,\psi_0)_{L_2(\Om)}=0.
\end{equation}
Then eigenfunction $\psi_0$ can be chosen so that
\begin{align}
(&\phi_0,\mathcal{T}\psi_0)_{L_2(\Om)}=1,\quad (\phi_0,\psi_0)_{L_2(\Om)}=0,\label{1.10a}
\\
&
\psi_0=\mathcal{P}\mathcal{T}\psi_0,\qquad\qquad \phi_0=\mathcal{P}\mathcal{T}\phi_0.\label{2.19}
\end{align}
Suppose then that this eigenfunction obeys the inequality
\begin{equation}\label{1.11}
\int\limits_{\G_+} \b\RE\psi_0\IM\psi_0\di x_1\not=0.
\end{equation}
Then for all sufficiently small $\e$ the operator $\Op_{\a+\e\b}$ has two simple isolated eigenvalues $\l_\e^\pm$ converging to $\l_0$ as $\e\to0$. These eigenvalues are real as
\begin{equation}\label{1.17}
\e\int\limits_{\G_+} \b\RE\psi_0\IM\psi_0\di x_1<0
\end{equation}
and are complex as
\begin{equation}\label{1.18}
\e\int\limits_{\G_+} \b\RE\psi_0\IM\psi_0\di x_1>0.
\end{equation}
Eigenvalues $\l_\e^\pm$ are holomorphic w.r.t. $\e^{1/2}$ and the first terms of their Taylor series read as
\begin{equation}\label{1.8}
\l_\e^\pm=\l_0+\e^{1/2}\l_{1/2}^\pm+\Odr(\e),
\quad
\l_{1/2}^\pm=\pm 2 \left(- \int\limits_{\G_+} \b\RE\psi_0\IM\psi_0\di x_1\right)^{1/2}.
\end{equation}
\end{theorem}

Let us discuss the results of these theorems. The typical situation of eigenvalues collision is that two simple eigenvalues of $\Op_{\a+\e\b}$ converge to the same limiting eigenvalue $\l_0$ of $\Op_\a$ as $\e\to0$. Then it is a general fact from the regular perturbation theory that the algebraic multiplicity of $\l_0$ should be two. The above theorems address two possible situations. In the first of them the geometric multiplicity of $\l_0$ is two, i.e., there exist two associated linearly independent eigenfunctions. As we see from Theorem~\ref{th1}, in this situation the perturbed eigenvalues are holomorphic w.r.t. $\e$ and their first terms in the Taylor series are given by (\ref{1.5}). The numbers $\l_1^\pm$ are some fixed constants and they can be both complex or real. But an important issue is that here under changing the sign of $\e$, the eigenvalues can not bifurcate from real line to the complex plane or vice versa. This fact is implied by (\ref{1.5}), namely, if $\l_1^\pm$ are complex numbers, then $\l_\e^\pm$ are also complex for both $\e<0$ and $\e>0$. Thus, in this case we do not face with the aforementioned phenomenon of eigenvalues collision discovered numerically in \cite{KS}, \cite{KT}.
If $\l_1^\pm$ are real, then we need to calculate the next terms of their Taylor series to see whether they are complex or real. Once all the terms in the Taylor series are real, we deal with two real eigenvalues which just pass one through the other staying on the real line. Nevertheless, in view of formulae (\ref{1.4b}) we believe that choosing appropriate $\b$ we can get almost any value for the quantity in (\ref{1.4a}). In a particular interesting case $\b=\a$ the author does not know a way of identifying the sign of $(b_{11}-b_{22})^2+4b_{12}^2$ or proving the reality of the eigenvalues $\l_\e^\pm$.

Theorem~\ref{th2} treats the case when the geometric multiplicity of $\l_0$ is one. Then the Taylor series for the perturbed eigenvalues are completely different in comparison with Theorem~\ref{th1} and here the expansions are made w.r.t. $\e^{1/2}$. And the presence of this power explains perfectly the studied phenomenon. Namely, once $\e$ is positive, the same is true for $\e^{1/2}$, while for negative $\e$ the square root $\e^{1/2}$ is pure imaginary. This is exactly what is needed, once $\e$ changes the sign, real eigenvalues become complex and vice versa. Unfortunately, we can not even analytically prove for our model the existence of such eigenvalues. We can just state that once $\l_0$ has a geometric multiplicity one and the associated eigenfunction $\psi_0$ satisfies the identity $(\psi_0,\mathcal{T}\psi_0)_{L_2(\Om)}=0$, then equation (\ref{1.6}) is solvable. And numerical results in \cite{KS}, \cite{KT} show that it is quite a typical situation.

Our next main result provide one more criterion identifying the solvability of equation (\ref{1.6}).

\begin{theorem}\label{th3}
Suppose $\psi_0$ is a simple eigenvalue of $\Op_\a$, the associated eigenfunction satisfies the estimate
\begin{equation}\label{1.12}
\sum\limits_{\genfrac{}{}{0 pt}{}{\g\in \mathds{Z}_+^2}{|\g|\leqslant 2}} \bigg|\frac{\p^\g\psi_0}{\p x^\g}(x)\bigg|\leqslant \frac{C}{1+|x_1|^3},\quad x\in\Om.
\end{equation}
Then equation (\ref{1.6}) is solvable if and only if
\begin{equation}\label{1.23}
\int\limits_{\mathds{R}^2} K(x_1,y_1)\big(\a(x_1)-\a(y_1)\big) \RE\psi_0(x_1,d)\IM\psi_0(y_1,d)\di x_1 \di y_1=0,
\end{equation}
where 
\begin{equation*}
K(x_1,y_1):=\left\{
\begin{aligned}
&x_1,\quad y_1<x_1,
\\
-&y_1,\quad y_1>x_1.
\end{aligned}\right.
\end{equation*}
Here $\psi_0$ is chosen so that it satisfies the first identity in (\ref{2.19}).
\end{theorem}

Assumption (\ref{1.12}) is not very restrictive since usually eigenfunctions associated with isolated eigenvalues of elliptic operators decay exponentially at infinity. The main condition here is (\ref{1.23}). As we shall show later in Lemma~\ref{lm2.1}, equation (\ref{1.6}) is solvable if and only if $(\psi_0,\mathcal{T}\psi_0)_{L_2(\Om)}=0$. And we rewrite this identity to  (\ref{1.23}) by calculating $(\psi_0,\mathcal{T}\psi_0)_{L_2(\Om)}$. The left hand side in (\ref{1.23}) is simpler in the sense that it involves only boundary integrals  while $(\psi_0,\mathcal{T}\psi_0)_{L_2(\Om)}$ is in fact the integral over whole the strip $\Om$.

\section{Proofs of main results}

In $L_2(\Om)$ we introduce the unitary operator $(\mathcal{U}_{\e\b}f)(x):=\E^{-\iu\e\b(x_1)x_2} f(x)$. Then it is easy to see that the spectra of $\Op_{\a+\e\b}$ and $\mathcal{U}_{\e\b}^{-1} \Op_{\a+\e\b}\mathcal{U}_{\e\b}$ coincide and
\begin{align}
& \mathcal{U}_{\e\b}^{-1} \Op_{\a+\e\b}\mathcal{U}_{\e\b} = \Op_\a-\e\mathcal{L}_\e,\label{2.1}
\\
&\mathcal{L}_\e:=-2\iu\b' x_2 \frac{\p\hphantom{x}}{\p x_1}-2\iu\b\frac{\p\hphantom{x}}{\p x_2}+\e\b^2-\e(\b')^2 x_2 - \iu\b'' x_2.\label{2.1a}
\end{align}
In the proofs of the main results we shall make use of several auxiliary lemmata.

\begin{lemma}\label{lm2.1}
Under the hypothesis of Theorem~\ref{th1} the equation
\begin{equation}\label{2.2}
(\Op_\a-\l_0)u=f
\end{equation}
is solvable if and only if
\begin{equation}\label{2.3}
(f,\mathcal{T}\psi_0)_{L_2(\Om)}=0.
\end{equation}
Under the hypothesis of Theorem~\ref{th2} equation (\ref{2.2}) is solvable if and only if
\begin{equation}\label{2.4}
(f,\mathcal{T}\psi_0^\pm)_{L_2(\Om)}=0.
\end{equation}
\end{lemma}

\begin{proof}
By (\ref{1.3}) we see that under the hypotheses of both Theorems~\ref{th1}~and~\ref{th2}, $\l_0$ is an eigenvalue of $\Op_\a^*$ with the associated eigenfunction(s) $\mathcal{T}\psi_0$ or $\mathcal{T}\psi_0^\pm$. Then the lemma follows from \cite[Ch. I\!I\!I, Sec. 6.6, Rem. 6.23]{Kato}.
\end{proof}

\begin{lemma}\label{lm2.4}
Suppose the hypothesis of Theorem~\ref{th2}. Then eigenfunction $\psi_0$ can be chosen so that relations (\ref{1.10a}), (\ref{2.19}), and
\begin{equation}\label{2.6a}
(\psi_0,\mathcal{T}\psi_0)_{L_2(\Om)}=0
\end{equation}
hold true. The functions $\RE\psi_0$ and $\RE\phi_0$ are even w.r.t. $x_2$ and $\IM\psi_0$ and $\IM\phi_0$ are odd w.r.t. $x_2$.
\end{lemma}

\begin{proof}
Identity (\ref{2.6a}) follows directly from (\ref{2.3}) applied to equation (\ref{1.6}). Since $\l_0$ is a real simple eigenvalue and equation (\ref{1.6}) has the unique solution
satisfying the second identity in (\ref{1.10a}), by (\ref{1.2}) we have (\ref{2.19}) and thus $\RE\psi_0$ and  $\RE\phi_0$ are even, while $\IM\psi_0$ and $\IM\phi_0$ are odd w.r.t. $x_2$. Employing this fact and (\ref{1.6}), we obtain
\begin{align*}
(\phi_0,\mathcal{T}\psi_0)_{L_2(\Om)}=&-\int\limits_{\Om} \phi_0(\D+\l_0)\phi_0\di x=\iu\int\limits_{\G_+} \a \phi_0^2 \di x_1-\iu\int\limits_{\G_-} \a \phi_0^2 \di x_1
\\
&+\int\limits_{\Om} \left(\left(\frac{\p \phi_0}{\p x_1}\right)^2+\left(\frac{\p \phi_0}{\p x_2}\right)^2-\l_0\phi_0^2\right)\di x
\\
=&-4\int\limits_{\G_+} \a \RE\phi_0\IM\phi_0\di x_1 + \int\limits_{\Om} \left( |\nabla \RE\phi_0|^2- |\nabla \IM\phi_0|^2\right)\di x
\\
&-\l_0\int\limits_{\Om} \left( |\RE\phi_0|^2- |\IM\phi_0|^2\right)\di x\in\mathds{R}.
\end{align*}
Hence, multiplying function $\psi_0$ and $\phi_0$ by an appropriate constant, we can easily get the first identity in (\ref{1.10a}) not spoiling other established properties of $\phi_0$ and $\psi_0$.
\end{proof}

\begin{lemma}\label{lm2.2}
Suppose the hypothesis of Theorem~\ref{th2}. Then for $\l$ close to $\l_0$ the resolvent $(\Op_\a-\l)^{-1}$ can be represented as
\begin{align}
&(\Op_\a-\l)^{-1}=\frac{\mathcal{P}_{-2}}{(\l-\l_0)^2} +\frac{\mathcal{P}_{-1}}{\l-\l_0} + \mathcal{R}_\a(\l),\label{2.5}
\\
&\mathcal{P}_{-2}=\psi_0 \ell_2, \quad \mathcal{P}_{-1}=\phi_0 \ell_2+\psi_0\ell_1,\nonumber
\\
&\ell_2 f:=- (f,\mathcal{T}\psi_0)_{L_2(\Om)},
\quad \ell_1 f:=- (f,\mathcal{T}\phi_0)_{L_2(\Om)},\label{2.9}
\end{align}
where $\mathcal{R}_\a(\l)$ is the reduced resolvent which is a bounded and holomorphic in $\l$ operator.
\end{lemma}

\begin{proof}
We know by \cite[Ch. I\!I\!I, Sec. 6.5]{Kato} (see also the remark on space $\mathbf{M}'(0)$ in the proof of Theorem~1.7 in \cite[Ch. V\!I\!I, Sec. 1.3]{Kato}) that
$(\Op_\a-\l)^{-1}$ can be expanded into the Laurent series
\begin{equation*}
(\Op_\a-\l)^{-1}=\sum\limits_{n=1}^{N} \frac{\mathcal{P}_{-n}}{(\l-\l_0)^n}+\mathcal{R}_\a(\l),
\end{equation*}
where $N$ is a fixed number independent of $\l$, $\mathcal{R}_\a$ is the reduced resolvent which is a bounded and holomorphic in $\l$ operator. Given any $f\in L_2(\Om)$, we then have
\begin{equation*}
u=(\Op_\a-\l)^{-1}f=\sum\limits_{n=1}^{N} \frac{u_{-n}}{(\l-\l_0)^n}+\sum\limits_{n=0}^{\infty} (\l-\l_0)^n u_n.
\end{equation*}
We substitute this formula into the equation $(\Op_\a-\l)u=f$ and equate the coefficients at the like powers of $(\l-\l_0)$:
\begin{align}
&(\Op_\a-\l_0)u_{-N}=0,\quad
(\Op_\a-\l_0)u_{-k}=u_{-k-1},\quad k=1,\ldots,N-1,\nonumber
\\
&(\Op_\a-\l_0)u_0=f+u_{-1},\quad (\Op_\a-\l_0)u_1=u_0.\label{2.8}
\end{align}
It implies that $u_{-N}=\psi_0\ell_2$, $u_{-N+1}=\phi_0\ell_2+\psi_0\ell_1$, where $\ell_i$ are some functionals on $L_2(\Om)$. If $N>2$, then by (\ref{1.7}) and Lemma~\ref{lm2.1} the equation for $u_{-N+2}$ is unsolvable. Hence, we can assume $N=2$.
Writing then the solvability condition (\ref{2.3}) for equations (\ref{2.8}) and taking into consideration the identity in (\ref{1.10a}), we arrive easily to the formula for $\ell_2$ in (\ref{2.9}) and
\begin{equation}\label{2.8a}
\ell_1 f:=- (U_0,\mathcal{T}\psi_0)_{L_2(\Om)},
\end{equation}
where  $U_0$ is the solution to the equation
\begin{equation}\label{2.6}
(\Op_\a-\l_0)U_0=f+\psi_0\ell_2 f
\end{equation}
satisfying
\begin{equation}\label{2.7}
(U_0,\psi_0)_{L_2(\Om)}=0.
\end{equation}
It follows from (\ref{1.3}) and (\ref{1.6}) that
\begin{align*}
(U_0,\mathcal{T}\psi_0)_{L_2(\Om)}=&(U_0,\mathcal{T}(\Op_\a-\l_0)\phi_0)_{L_2(\Om)} = \big(U_0,(\Op_\a-\l_0)^*\mathcal{T}\phi_0\big)_{L_2(\Om)}
\\
=&\big((\Op_\a-\l_0)U_0,\mathcal{T}\phi_0\big)_{L_2(\Om)} =\big(f+\psi_0\ell_2f,\mathcal{T}\psi_0\big)_{L_2(\Om)}.
\end{align*}
These identities and (\ref{2.6a}), (\ref{2.8a}) imply formula (\ref{2.6}) for $\ell_1$.
\end{proof}

\begin{lemma}\label{lm2.3}
Suppose the hypothesis of Theorem~\ref{th1}.
Then for $\l$ close to $\l_0$ the resolvent $(\Op_\a-\l)^{-1}$ can be represented as
\begin{gather}
(\Op_\a-\l)^{-1}=\frac{\mathcal{P}_{-1}}{\l-\l_0} + \mathcal{R}_\a(\l),\label{2.10}
\\
 \mathcal{P}_{-1}=\psi_0^+ \ell_+ + \psi_0^-\ell_-,\quad
\ell_\pm f:=- (f,\mathcal{T}\psi_0^\pm)_{L_2(\Om)},\label{2.11}
\end{gather}
where $\mathcal{R}_\a(\l)$ is the reduced resolvent which is a bounded and holomorphic in $\l$ operator.
\end{lemma}

The proof of this lemma is similar to that of Lemma~\ref{lm2.2}, we just should bear in mind that due to (\ref{1.9}) and Lemma~\ref{lm2.1} the equations
\begin{equation*}
(\Op_\a-\l_0)u=\psi_0^\pm
\end{equation*}
are unsolvable.

We proceed to the proofs of Theorems~\ref{th1},~\ref{th2},~\ref{th3}.

\begin{proof}[Proof of Theorem~\ref{th2}] The proof is based on the modified version of Birman-Schwinger principle suggested in \cite{G} in the form developed in \cite{MSb}. In view of (\ref{2.1}), the eigenvalue equation for $\Op_{\a+\e\b}$ is equivalent to
the same equation for $\Op_\a-\e\mathcal{L}_\e$. The latter equation can be written as
\begin{equation}\label{2.12}
(\Op_\a-\l_\e)\psi_\e=\e\mathcal{L}_\e\psi_\e.
\end{equation}
We then invert the operator $(\Op_\a-\l_\e)$ by Lemma~\ref{lm2.2} and obtain
\begin{equation*}
\psi_\e=\e\frac{\mathcal{P}_{-2}\mathcal{L}_\e\psi_\e}{(\l_\e-\l_0)^2}+\e \frac{\mathcal{P}_{-1}\mathcal{L}_\e\psi_\e}{\l_\e-\l_0}+\e\mathcal{R}_\a(\l_\e)\psi_\e.
\end{equation*}
By Lemma~\ref{lm2.2} the operator $\mathcal{R}_\a(\l)$ is bounded uniformly in $\l$ close to $\l_0$ and hence the inverse $\mathcal{A}(z,\e):=\big(\I-\e \mathcal{R}_\a(\l_0+z)\big)^{-1}$ is well-defined and is uniformly bounded for all $\l$ close to $\l_0$ and for all sufficiently small $\e$. We apply this operator to the latter equation and get
\begin{equation}\label{2.13}
\psi_\e=\frac{\e}{z_\e^2} \mathcal{A}(\l_0+z_\e,\e)
\mathcal{P}_{-2}\mathcal{L}_\e\psi_\e+ \frac{\e}{z_\e}\mathcal{A}(\l_0+z_\e,\e) \mathcal{P}_{-1}\mathcal{L}_\e\psi_\e,
\end{equation}
where we denote $z_\e:=\l_\e-\l_0$.
Then we apply functionals $\ell_2\mathcal{L}_\e$, $\ell_1\mathcal{L}_\e$ to the obtained equation and it results in
\begin{equation}\label{2.15}
\begin{aligned}
&\left(\frac{\e}{z_\e}A_{11}(z_\e,\e)-1\right)X_1 + \frac{\e}{z_\e^2} \big(A_{11}(z_\e,\e)+z_\e A_{12}(z_\e,\e)\big)X_2=0,
\\
&\hphantom{((}\frac{\e}{z_\e} A_{21}(z_\e,\e)X_1+ \left(\frac{\e}{z_\e^2}\big(A_{21}(z_\e,\e)+z_\e A_{22}(z_\e,\e)\big)-1\right)X_2=0,
\end{aligned}
\end{equation}
where $X_i=\ell_i\mathcal{L}_\e\psi_\e$, and
\begin{equation*}
A_{i1}(z,\e):=\ell_i \mathcal{L}_\e \mathcal{A}(\l_0+z,\e)\psi_0,\quad A_{i2}(z,\e):=\ell_i \mathcal{L}_\e \mathcal{A}(\l_0+z,\e)\phi_0,\quad i=1,2.
\end{equation*}
The obtained system of equations is linear w.r.t. $(X_1,X_2)$. We need a non-zero solution to this system since otherwise by (\ref{2.13}) we would get $\psi_\e=0$ and $\psi_\e$ then can not be an eigenfunction. System (\ref{2.15}) has a nonzero solution if its determinant vanishes. It implies the equation
\begin{align*}
z_\e^2&-\e\big(A_{11}(z_\e,\e)+A_{22}(z_\e,\e))z_\e
\\
&-\e A_{21}(z_\e,\e)+\e^2\big(A_{11}(z_\e,\e)A_{22}(z_\e,\e) -A_{12}(z_\e,\e)A_{21}(z_\e,\e)\big)=0,
\end{align*}
which is equivalent to the following two
\begin{equation}\label{2.14}
z_\e=G_\pm(z_\e,\e^{1/2}),
\end{equation}
where
\begin{equation}\label{2.20}
\begin{aligned}
&G_\pm(z,\k):=
\frac{\k^2\big(A_{11}(z,\k^2)+A_{22}(z,\k^2)\big)}{2}
\\
&\pm \k \bigg(A_{21}(z,\k^2)
+\frac{\e}{4}\big(A_{11}(z,\k^2)-A_{22}(z,\k^2)\big)^2
+\e A_{12}(z,\k^2) A_{21}(z,\k^2)
\bigg)^{1/2}.
\end{aligned}
\end{equation}
Here the branch of the square root is fixed by the restriction $1^{1/2}=1$. It is clear that the functions $A_{ij}$ are jointly holomorphic w.r.t. sufficiently small $z$ and $\e$. Moreover, by (\ref{2.1a})
\begin{equation}\label{2.17}
A_{21}(0,\e)=\ell_2\mathcal{L_\e}\mathcal{A}(0,\e)\psi_0=\iu \ell_2 \left(-2\b'x_2\frac{\p\hphantom{x}}{\p x_1}-2\b\frac{\p\hphantom{x}}{\p x_2}-\b'' x_2\right)\psi_0+\Odr(\e).
\end{equation}
To calculate the first term in the right hand side of this identity, we first observe that by the equation for $\psi_0$ we have
\begin{equation*}
-\left(2\b'x_2\frac{\p\hphantom{x}}{\p x_1}+2\b\frac{\p\hphantom{x}}{\p x_2}+\b'' x_2\right)\psi_0 = -(\D+\l_0)\b x_2\psi_0=:g.
\end{equation*}
Now we find $\iu\ell_2 g$ by integration by parts
\begin{equation}\label{2.18}
\begin{aligned}
\iu \ell_2 g=& \int\limits_{\Om} \psi_0 (\D+\l_0) \b x_2\psi_0\di x
= \iu\int\limits_{\G_+} \left(\psi_0\frac{\p\hphantom{x}}{\p x_2}\b x_2\psi_0-\b x_2\psi_0\frac{\p\psi_0}{\p x_2}\right) \di x_1
\\
& - \iu\int\limits_{\G_-} \left(\psi_0\frac{\p\hphantom{x}}{\p x_2}\b x_2\psi_0-\b x_2\psi_0\frac{\p\psi_0}{\p x_2}\right) \di x_1
= \iu\int\limits_{\G_+} \b\psi_0^2 \di x_1 - \iu\int\limits_{\G_-} \b\psi_0^2 \di x_1.
\end{aligned}
\end{equation}
 Together with (\ref{2.18}) it implies
\begin{equation}\label{2.21}
\iu\ell_2 g=-  4 \int\limits_{\G_+} \b\RE\psi_0\IM\psi_0 \di x_1.
\end{equation}
Hence, by (\ref{2.20}), (\ref{2.18}), (\ref{1.11}), and the properties of functions $A_{ij}$ we conclude that functions $G_\pm$ are jointly holomorphic w.r.t. sufficiently small $z$ and $\k$. Applying then Rouch\'e theorem as it was done in \cite[Sec. 4]{MSb}, we conclude that for all sufficiently small $\k$ each of the functions $z\mapsto z-G_\pm(z,\k)$ has a simple zero $z_\pm(\k)$ in a small neighborhood of the origin. By the implicit function theorem these zeroes are holomorphic w.r.t. $\k$. Thus, the desired solutions to equations (\ref{2.14}) are $z_\pm(\e^{1/2})$, and these functions are holomorphic w.r.t. $\e^{1/2}$. Moreover, it follows from (\ref{2.14}), (\ref{2.20}), (\ref{2.17}), (\ref{2.18}), (\ref{2.21}) that
\begin{equation*}
z_\pm(\e^{1/2})=G_\pm(0,\e^{1/2})+\Odr(\e)=\pm \e^{1/2} A_{21}^{1/2}(0,\e)+\Odr(\e)
\end{equation*}
and then the sought eigenvalues are $\l_\e^\pm=\l_0+z_\pm(\e^{1/2})$. These eigenvalues are holomorphic w.r.t. $\e^{1/2}$ and obey (\ref{1.8}). Let us prove that these eigenvalues are real as (\ref{1.17}) holds true and are complex once (\ref{1.18}) is satisfied. The latter statement follows easily from formulae (\ref{1.8}) since in this case $\e^{1/2}\l_{1/2}^\pm$ are two imaginary numbers. To prove the reality, as one can easily make sure, it is sufficient to prove that functions $G_\pm(z,\k)$ are real for real $z$ and $\k$. Then the existence of a real root is implied easily by the implicit function theorem for real functions.

In view of definition (\ref{2.20}) of $G_\pm$, the desired fact is yielded by the similar reality of $A_{ij}$. Let us prove the latter.

It follows from Lemma~\ref{lm2.3} that for each $f\in L_2(\Om)$ the function
\begin{equation*}
\mathcal{R}_\a(\l)f=(\Op_\a-\l)^{-1}-\frac{\mathcal{P}_{-2}}{(\l-\l_0)^2} -\frac{\mathcal{P}_{-1}}{\l-\l_0}
\end{equation*}
solves the equation
\begin{equation}\label{2.29}
(\Op_\a-\l)\mathcal{R}_\a(\l)f=f+\psi_0\ell_1 f+ \phi_0\ell_2 f.
\end{equation}
Employing definition (\ref{2.1a}) of $\mathcal{L}_\e$, we check easily that $\mathcal{PT}\mathcal{L}_\e= \mathcal{L}_\e\mathcal{PT}$. This identity and (\ref{2.19}), (\ref{2.29}) yield that for $z\in\mathds{R}$, $\k\in\mathds{R}$
\begin{equation*}
\mathcal{PT}\mathcal{L}_\e \mathcal{A}(\l_0+z,\k)\psi_0=\mathcal{L}_\e \mathcal{A}(\l_0+z,\k)\psi_0,\quad \mathcal{PT}\mathcal{L}_\e \mathcal{A}(\l_0+z,\k)\phi_0=\mathcal{L}_\e \mathcal{A}(\l_0+z,\k)\phi_0.
\end{equation*}
Using  (\ref{2.19}) once again, for $z\in\mathds{R}$, $\k\in\mathds{R}$ we get
\begin{align*}
\overline{A_{11}(z,\k)}=&\big(\mathcal{PT}\mathcal{L}_\e \mathcal{A}(\l_0+z,\k)\psi_0,\mathcal{P}\psi_0\big)_{L_2(\Om)}
\\
=&\big(\mathcal{T}\mathcal{L}_\e \mathcal{A}(\l_0+z,\k)\psi_0,\mathcal{T}\psi_0\big)_{L_2(\Om)}=A_{11}(z,\k).
\end{align*}
The reality of other functions $A_{ij}$ can be proven in the same way. The proof is complete.
\end{proof}

\begin{proof}[Proof of Theorem~\ref{th1}] The main ideas here are the same as in the proof of Theorem~\ref{th2}, so, we focus only on the main milestones. We again begin with (\ref{2.1}) and invert $(\Op_\e-\l_\e)$ by Lemma~\ref{lm2.4}. It leads us  to an analogue of equation (\ref{2.13}),
\begin{equation}\label{2.25}
\psi_\e=\frac{\e}{z_\e} \mathcal{A}(\l_0+z_\e,\e)\mathcal{P}_{-1}\mathcal{L}_\e\psi_\e,
\end{equation}
where the operator $\mathcal{A}$ is introduced in the same way as above. We apply then functionals $\ell_\pm \mathcal{L}_\e$ to this equation
\begin{gather}
\begin{aligned}
&\left(\frac{\e}{z_\e}B_{11}(z_\e,\e)-1\right)X_1 + \frac{\e}{z_\e} B_{12}(z_\e,\e) X_2=0,
\\
&\hphantom{11} \frac{\e}{z_\e}B_{21}(z_\e,\e)X_1 +\left(\frac{\e}{z_\e}B_{22}(z_\e,\e)-1\right) X_2=0,
\end{aligned}\label{2.26}
\\
\begin{aligned}
&B_{11}(z,\e):=\ell_+ \mathcal{L}_\e \mathcal{A}(\l_0+z,\e)\psi_0^+, \quad B_{12}(z,\e):=\ell_+ \mathcal{L}_\e \mathcal{A}(\l_0+z,\e)\psi_0^-,
\\
&B_{21}(z,\e):=\ell_- \mathcal{L}_\e \mathcal{A}(\l_0+z,\e)\psi_0^+, \quad B_{22}(z,\e):=\ell_- \mathcal{L}_\e \mathcal{A}(\l_0+z,\e)\psi_0^-.
\end{aligned}
\nonumber
\end{gather}
The determinant of system (\ref{2.26}) should again vanish and it implies the equation
\begin{equation*}
z_\e^2-\e\big(B_{11}(z_\e,\e)+B_{22}(z_\e,\e)\big) +\e^2\big(B_{11}(z_\e,\e)B_{22}(z_\e,\e)-B_{12}(z_\e,\e)B_{21}(z_\e,\e)\big)=0,
\end{equation*}
which splits into other two
\begin{gather}\label{2.27}
z_\e=Q_\pm(z_\e,\e),
\\
\begin{aligned}
Q_\pm(z,\e):=&\frac{\e}{2}\big(B_{11}(z_\e,\e)+B_{22}(z_\e,\e)\big) \\
&\pm \frac{\e}{2} \Big((B_{11}(z,\e)-B_{22}(z,\e))^2+4 B_{12}(z,\e) B_{21}(z,\e)\Big)^{1/2}.
\end{aligned}
\nonumber
\end{gather}
Here the branch of the square root is fixed by the restriction $1^{1/2}=1$. Let us prove that this square root is jointly holomorphic w.r.t. $z$ and $\e$. Integrating by parts as in (\ref{2.18}) and employing (\ref{1.1}), one can make easily sure that
\begin{equation}\label{2.28}
\begin{aligned}
B_{ii}=b_{ii}+\Odr(\e),\quad i=1,2,\quad B_{12}(0,\e)=b_{12}+\Odr(\e),\quad B_{21}(0,\e)=b_{21}+\Odr(\e).
\end{aligned}
\end{equation}
Hence, by assumption (\ref{1.4a}), functions $Q_\pm$ are jointly holomorphic w.r.t. $z$ and $\e$. Proceeding now as in the proof of Theorem~\ref{th2}, we arrive at the statement of Theorem~\ref{th1}.
\end{proof}

\begin{proof}[Proof of Theorem~\ref{th3}] Denote 
\begin{equation*}
\psi(x):=\frac{1}{2}x_1\int\limits_{-\infty}^{x_1} t\psi_0(t,x_2)\di t.
\end{equation*}
In view of (\ref{1.12}) this function is well-defined. Throughout the proof we shall deal with several integrals of such kind and all of them will be well-defined due to (\ref{1.12}). In what follows we shall not stress this fact anymore.

Employing the equation for $\psi_0$, integrating by parts, and bearing in mind estimates (\ref{1.12}), we get
\begin{align*}
(\D+\l_0)\psi =&\psi_0+\frac{1}{2}x_1\frac{\p\psi_0}{\p x_1} + \frac{1}{2}x_1\int\limits_{-\infty}^{x_1} t\left(\frac{\p^2\hphantom{t}}{\p x_2^2}+\l_0\right)\psi_0(t,x_2)\di t
\\
=&\psi_0+\frac{1}{2}x_1\frac{\p\psi_0}{\p x_1} -
\frac{1}{2} x_1\int\limits_{-\infty}^{x_1} \frac{\p^2\psi_0}{\p x_1^2}(t, x_2)\di t=\psi_0.
\end{align*}
The proven equation for $\psi$ allows us to integrate once again,
\begin{align*}
\int\limits_{\Om} \psi_0^2\di x=&\int\limits_{\Om} \psi_0(\D+\l_0) \psi\di x= \int\limits_{\G_+} \left(\psi_0\frac{\p\psi}{\p x_2}-\psi\frac{\p\psi_0}{\p x_2}\right)\di x_1-\int\limits_{\G_-} \left(\psi_0\frac{\p\psi}{\p x_2}-\psi\frac{\p\psi_0}{\p x_2}\right)\di x_1
\\
=&\int\limits_{\G_+} \psi_0\left(\frac{\p\psi}{\p x_2}+\iu\a\psi\right)\di x_1-\int\limits_{\G_-} \psi_0 \left(\frac{\p\psi}{\p x_2}+\iu\a\psi\right)\di x_1.
\end{align*}
Now we employ identity (\ref{2.19}) and boundary condition (\ref{1.1}) for $\psi_0$ to simplify the sum of these integrals,
\begin{align*}
\int\limits_{\Om} \psi_0^2\di x=&-\int\limits_{\G_+} \di x_1 \RE\psi_0(x_1,d) x_1 \int\limits_{-\infty}^{x_1} \big(\a(x_1)-\a(y_1)\big)\IM\psi_0(y_1,d)\di y_1 
\\
&-\int\limits_{\G_+} \di x_1 \IM\psi_0(x_1,d) x_1 \int\limits_{-\infty}^{x_1} \big(\a(x_1)-\a(y_1)\big)\RE\psi_0(y_1,d)\di y_1
\\
=&-\int\limits_{\G_+} \di x_1 \RE\psi_0(x_1,d) x_1 \int\limits_{-\infty}^{x_1} \big(\a(x_1)-\a(y_1)\big)\IM\psi_0(y_1,d)\di y_1
\\
&+ \int\limits_{\G_+} \di x_1 \RE\psi_0(x_1,d) x_1 \int\limits_{x_1}^{+\infty} \big(\a(x_1)-\a(y_1)\big)\IM\psi_0(y_1,d)\di y_1
\\
=&-\int\limits_{\mathds{R}^2} K(x_1,y_1) \big(\a(x_1)-\a(y_1)\big) \RE\psi_0(y_1,d)\IM\psi_0(y_1,d)\di x_1 \di y_1.
\end{align*}
By (\ref{2.3}) we then conclude that equation (\ref{1.6}) is solvable if and only if identity (\ref{1.23}) holds true.
\end{proof}
\begin{remark}
The idea of the latter proof was borrowed from the proof of Lemma~2.2 in \cite{AsAn03}, see also proof of Lemma~3.6 in \cite{MSb}.
\end{remark}

\section*{Acknowledgments}

The work is partially supported by RFBR, by a grant of the President of Russia for young scientists - doctors of science (MD-183.2014.1) and by the fellowship of Dynasty foundation for young mathematicians.

\end{document}